\documentclass[12pt]{article}
\usepackage{amsfonts}
\usepackage[a4paper,margin=3cm]{geometry}
\usepackage{amsmath}
\usepackage{amssymb}

\newtheorem{theorem}{Theorem}[section]

\newtheorem{corollary}[theorem]{Corollary}
\newtheorem{proposition}[theorem]{Proposition}

\newenvironment{proof}{{\bf Proof.}}{\hfill$\Box$\\}
\newenvironment{remark}{{\vskip 1ex\noindent\bf Remark.}}{\\}

\newcommand{\R}{\mathbb{R}}

\newcommand{\spn}{\mathrm{span}}

\newcommand{\DD}{\mathcal{D}}
\newcommand{\GG}{\mathcal{G}}
\newcommand{\HH}{\mathcal{H}}
\newcommand{\LL}{\mathcal{L}}
\newcommand{\nc}{\nabla^c}

\begin{document}

\title{\textbf{Invariants of contact sub-pseudo-Riemannian structures and Einstein-Weyl geometry}}
\author{ Marek Grochowski\thanks{
\textbf{Faculty of Mathematics and Natural Sciences, Cardinal Stefan Wyszy\'nski University, ul.~Dewajtis 5, 01-815 Warszawa, Poland}
\newline
E-mail: m.grochowski@uksw.edu.pl}
\and Wojciech Kry\'nski\thanks{
\textbf{Institute of Mathematics, Polish Academy of Sciences, ul.~\'Sniadeckich 8, 00-956 Warszawa, Poland}
\newline
E-mail: krynski@impan.pl}}
\maketitle

\begin{abstract}
We consider local geometry of sub-pseudo-Riemannian structures on contact manifolds. We construct fundamental invariants of the structures and show that the structures give rise to Einstein-Weyl geometries in dimension 3, provided that certain additional conditions are satisfied.
\end{abstract}

\section{Introduction}

A sub-pseudo-Riemannian contact manifold $(M,\DD,g)$ is an odd-dimensional manifold $M$ endowed with a pair $(\DD,g)$ where $\DD$ is a contact distribution on $M$ and $g$ is a pseudo-Riemannian metric on $\DD$. In particular, if $g$ is positive-definite then $(\DD,g)$ is called contact sub-Riemannian structure, and if $g$ has signature $(-,+,\ldots ,+)$ then $(\DD,g)$ is called contact sub-Lorentzian structure. In the present paper we study local geometry of sub-pseudo-Riemannian contact manifolds. All objects are assumed to be smooth. It is well known, that any contact manifold $(M,\DD)$ of dimension $2n+1$ is locally diffeomorphic to the $(2n+1)$-dimensional Heisenberg group, i.e.\ there exist local coordinates $(x_1,\ldots ,x_n,y_1,\ldots,y_n,z)$ on $M$ such that 
$$
\DD=\spn\left\{\frac{\partial}{\partial x_i}-\frac{1}{2}y_i\frac{\partial}{\partial z},\ \frac{\partial}{\partial y_i}+\frac{1}{2}x_i\frac{\partial}{\partial z}\ |\ i=1,\ldots ,n\right\}.
$$
Therefore, we can assume that there is given the Heisenberg group equipped with an additional metric on $\DD$. If the metric is left-invariant with respect to the action of the Heisenberg group then the structure is called flat.

The contact sub-pseudo-Riemannian structures appear in control theory. For instance, many efforts have been made aiming to analyse the behaviour of the sub-pseudo-Riemannian geodesics \cite{A,G,HY,KM} or constructing normal forms \cite{ACG,G}. Other applications can be found in \cite{GV,M}. The geometry of the structures is well understood in dimension 3 only. In this dimension invariants have been constructed in \cite{A,AB} for the Riemannian signature and in \cite{GW} for the Lorentzian signature. An alternative approach to the equivalence problem in dimension 3 is proposed in \cite{GMW} and it uses Tanaka's theory of graded, nilpotent Lie algebras. In the present paper we present yet another approach and generalise it to higher dimensions. We construct a system of invariants of the contact sub-pseudo-Riemannian structures in any dimension and any signature (Theorems \ref{thm3} and \ref{thm4}). We also show that there are connections between sub-pseudo-Riemannian geometries and Einstein-Weyl structures \cite{D,H,JT} in dimension 3 (Theorems \ref{thmEW1} and \ref{thmEW2}). The last section contains remarks on isometries of the considered structures.

The main idea of the paper is to consider extensions of $g$ on $\DD$ to metrics on the tangent bundle $TM$. It can be done in a canonical way, because there is a well defined Reeb vector field on $M$ which is transversal to $\DD$. We use the Levi-Civita connections of the extended metrics and show that the corresponding curvature tensors contain all basic invariants of the original structure. Moreover, a more detailed analysis proves that the extended metrics give rise to Einstein-Weyl structures provided that additional conditions are satisfied. This gives an alternative construction of certain classical examples of the Einstein-Weyl structures obtained before in \cite{JT,PT} or \cite{DK}. Indeed, we extend a sub-pseude-Riemannian metric, rather than reduce a metric on a four-dimensional manifold.

\section{Dimension 3}

\paragraph{Structural functions.}
Let $\DD$ be a contact distribution on a 3-dimensional manifold $M$ and let $g$ be a metric on $\DD$. We consider two cases depending on the signature of $g$ i.e.\ $(+,+)$ or $(-,+)$. If $g$ is positive-definite we say that $(\DD,g)$ is a sub-Riemannian structure on $M$ and if $g$ is indefinite we say that $(\DD,g)$ is a sub-Lorentzian structure on $M$. We shall assume that there is given an orientation of $\DD$. If it is the case then $(\DD,g)$ is referred to as an oriented contact sub-pseudo-Riemannian structure.

Let us consider a local, positively oriented, orthonormal frame $(X_1,X_2)$ of $\DD$. In the sub-Riemannian case 
$$
g(X_1,X_1)=1,\qquad g(X_2,X_2)=1,\qquad g(X_1,X_2)=0.
$$
In the sub-Lorentzian case we assume that $X_1$ is unit time-like and $X_2$ is unit space-like, i.e. 
$$
g(X_1,X_1)=-1,\qquad g(X_2,X_2)=1,\qquad g(X_1,X_2)=0.
$$
Thus, in the matrix form 
$$
g=\left( 
\begin{array}{cc}
1 & 0\\ 
0 & 1
\end{array}
\right),
$$
or 
$$
g=\left( 
\begin{array}{cc}
-1 & 0\\ 
0 & 1
\end{array}
\right),
$$
depending on the signature.

Note that the frame $(X_1,X_2)$ is complemented to the full frame on $M$ by the Reeb vector field $X_0$. By definition $X_0$ is a vector field such that the Lie bracket $[X_0,X_i]$ is a section of $\DD$, for $i=1,2$, and $[X_1,X_2]=X_0\mod\DD$. Thus we can write
\begin{align*}
\left[X_0, X_1\right]&=c_{01}^1X_1+c_{01}^2X_2, \\
\left[X_0, X_2\right]&=c_{02}^1X_1+c_{02}^2X_2, \\
\left[X_1, X_2\right]&=c_{12}^1X_1+c_{12}^2X_2+X_0,
\end{align*}
and the coefficients $c_{ij}^k$ are referred to as the structural functions of the frame $(X_1,X_2)$. The triple $(X_1,X_2,X_0)$ defines an orientation on $M$, which is induced by the original orientation of $\DD$.

\paragraph{Invariants.}
There are two fundamental invariants of the structure $(\DD,g)$ which were defined in \cite{A} in the sub-Riemannian case and in \cite{GW} in the sub-Lorentzian case. The invariants depend on the chosen orientation of $\DD$. The first invariant, denoted $h$, can be thought of as a bi-linear form on $\DD$ and it is defined by the formula
\begin{equation}\label{hdim3}
h=\frac{1}{2}\LL_{X_0}g,
\end{equation}
where $\LL_{X_0}$ denotes the Lie derivative in the direction of $X_0$. It is well defined because the flow of the Reeb field $X_0$ preserves $\DD$, i.e. $[X_0,\DD]=\DD$. The matrix form of $h$ in the orthonormal frame takes the form 
$$
h=-\left( 
\begin{array}{cc}
c_{01}^1 & \frac{1}{2}\gamma\\ 
\frac{1}{2}\gamma & c_{02}^2
\end{array}
\right)
$$
in the sub-Riemannian case, where $\gamma=c_{01}^2+c_{02}^1$, and the form 
$$
h=-\left( 
\begin{array}{cc}
-c_{01}^1 & \frac{1}{2}\gamma\\ 
\frac{1}{2}\gamma & c_{02}^2
\end{array}
\right)
$$
in the sub-Lorentzian case, where $\gamma=c_{01}^2-c_{02}^1$.

The second invariant, denoted $\kappa$, is given by the formula 
\begin{equation}\label{kappasr}
\kappa = X_1(c_{12}^2)-X_2(c_{12}^1)-(c_{12}^1)^2-(c_{12}^2)^2+\frac{1}{2}(c_{01}^2-c_{02}^1)
\end{equation}
in the sub-Riemannian case and by the formula 
\begin{equation}\label{kappasl}
\kappa = X_1(c_{12}^2)+X_2(c_{12}^1)+(c_{12}^1)^2-(c_{12}^2)^2+\frac{1}{2}(c_{01}^2+c_{02}^1)
\end{equation}
in the sub-Lorentzian case. It is known that in the both cases $\kappa$ can be interpreted as a curvature of the system \cite{A, GMW}.
\footnote{Note that our structural functions and the Reeb vector field differs by sign with respect to the analogous objects in \cite{A,GW}. Therefore our formula for $\kappa$ is slightly different than the formulae in \cite{A,GW}. However, it is the same invariant.}

It will be convenient to consider the endomorphism $h^\sharp\colon\DD\to\DD$ 
$$
g(h^\sharp(X),Y)=h(X,Y)
$$
instead of the bi-linear form $h$. In the matrix form $h^\sharp=g^{-1}h$.

\paragraph{Canonical extensions.}
Let $c\in\R\setminus\{0\}$ be a fixed constant. Any metric $g$ on $\DD$ extends uniquely to a metric $G^c$ on $M$ such that 
$$
G^c(X_0,X_0)=c,
$$
and 
$$
G^c(X_0,X_i)=0, \qquad i=1,2.
$$
If the original metric $g$ is indefinite then, clearly, $G^c$ is indefinite too. If $g$ is positive-definite then $G^c$ is positive definite for $c>0$ and indefinite for $c<0$. Moreover, the extension $G^c$ does not depend on the orientation of $\DD$. Indeed, a change of the orientation implies that $X_0$ is multiplied by $-1$ and this does not affect $G^c$. It follows that the extension $G^c$ is canonically defined by the original structure $(\DD,g)$ and the constant $c$. Therefore the associated Levi-Civita connection $\nc$ and the curvature tensor of $\nc$ are defined by the structure $(\DD,g)$ itself.

The connection $\nc$ can be easily computed in terms of the structural functions. In the sub-Riemannian case we get 
\begin{align*}
\nc_{X_0}X_0&=0,\\
\nc_{X_0}X_1&=c_{01}^2X_2-\frac{1}{2}(\gamma+c)X_2,\\
\nc_{X_0}X_2&=c_{02}^1X_1-\frac{1}{2}(\gamma-c)X_1,\\
\nc_{X_1}X_0&=-c_{01}^1X_1-\frac{1}{2}(\gamma+c)X_2,\\
\nc_{X_1}X_1&=\frac{1}{c}c_{01}^1X_0-c_{12}^1X_2,\\
\nc_{X_1}X_2&=\frac{1}{2c}(\gamma+c)X_0+c_{12}^1X_1,\\
\nc_{X_2}X_0&=-\frac{1}{2}(\gamma-c)X_1-c_{02}^2X_2,\\
\nc_{X_2}X_1&=\frac{1}{2c}(\gamma-c)X_0-c_{12}^2X_2,\\
\nc_{X_2}X_2&=\frac{1}{c}c_{02}^2X_0+c_{12}^2X_1,
\end{align*}
where $\gamma=c_{01}^2+c_{02}^1$. In the sub-Lorentzian case we get 
\begin{align*}
\nc_{X_0}X_0&=0,\\
\nc_{X_0}X_1&=c_{01}^2X_2-\frac{1}{2}(\gamma+c)X_2,\\
\nc_{X_0}X_2&=c_{02}^1X_1+\frac{1}{2}(\gamma-c)X_1,\\
\nc_{X_1}X_0&=-c_{01}^1X_1-\frac{1}{2}(\gamma+c)X_2,\\
\nc_{X_1}X_1&=-\frac{1}{c}c_{01}^1X_0+c_{12}^1X_2,\\
\nc_{X_1}X_2&=\frac{1}{2c}(\gamma+c)X_0+c_{12}^1X_1,\\
\nc_{X_2}X_0&=\frac{1}{2}(\gamma-c)X_1-c_{02}^2X_2,\\
\nc_{X_2}X_1&=\frac{1}{2c}(\gamma-c)X_0-c_{12}^2X_2,\\
\nc_{X_2}X_2&=\frac{1}{c}c_{02}^2X_0-c_{12}^2X_1,
\end{align*}
where $\gamma=c_{01}^2-c_{02}^1$. The formulae implies the following
\begin{theorem}\label{thm1}
Let $(\DD,g)$ be a contact sub-Riemannian or sub-Lorentzian structure on a three-dimensional manifold. Then the sectional curvature $\kappa^c_\DD$ of $\DD$ with respect to the extended metric $G^c$ equals to 
\begin{equation}\label{kappac}
\kappa^c_\DD=\kappa-\frac{1}{c}\det(h^\sharp)-\frac{3c}{4}.
\end{equation}
\end{theorem}
\begin{proof}
Let $(X_1,X_2)$ be an orthonormal frame of $\DD$. The formula is obtained by direct computations of $G^c(R(X_1,X_2)X_2,X_1)=G^c(\nc_{X_1}\nc_{X_2}X_2-\nc_{X_2}\nc_{X_1}X_2-\nc_{[X_1,X_2]}X_2,X_1)$ using \eqref{kappasr}, \eqref{kappasl}, the formulae for $\nc$ provided above and the following identities 
$$
\det(h^\sharp)=c_{01}^1c_{02}^2-\frac{1}{4}\gamma^2,\qquad \frac{1}{2}\gamma-c_{01}^2=\frac{1}{2}(c_{02}^1-c_{01}^2)
$$
in the sub-Riemannian case and 
$$
\det(h^\sharp)=c_{01}^1c_{02}^2+\frac{1}{4}\gamma^2,\qquad -\frac{1}{2}\gamma+c_{01}^2=\frac{1}{2}(c_{02}^1+c_{01}^2)
$$
in the sub-Lorentzian case.
\end{proof}

\begin{remark}
The formula for $\kappa_\DD^c$ also holds if $c$ is not a constant but a function on $M$. Indeed, if $c$ depends on a point in $M$ then $\nc_{X_i}X_0$ and $\nc_{X_0}X_i$, where $i=1,2$, are modified by a term of the form $\frac{1}{2c}X_i(c)X_0$. However, the modification does not affect $\kappa_\DD^{c}$. In particular if 
$$
\frac{3}{4}c^2+\det(h^\sharp)=0
$$
then \eqref{kappac} reduces to $\kappa_\DD^{c}=\kappa$.
\end{remark}

\begin{remark}
The invariant $\kappa$ was defined in \cite{GW} by formula \eqref{kappasl} as an invariant of a time- and space-oriented sub-Lorentzian structure. Theorem \ref{thm1} permits to regard $\kappa$ as an invariant of a sub-Lorentzian structure without any orientation. Indeed, our choice of the orientation does not affect neither the metric $G^c$ nor the sectional curvature $\kappa_{\DD}^c$.

On the other hand, the invariant $h$ defined by \eqref{hdim3} depends essentially on $X_0$ and thus on the orientation of $\DD$. However, the condition $h=0$ is independent of the choice of the orientation.
\end{remark}

\begin{remark}
Observe that no matter the value of $c$ is, the trajectories of the Reeb field are geodesics for the metric $G^{c}$.
\end{remark}

\paragraph{Symmetric case.}
Assume that $h=0$. It means that the Reeb vector field is an infinitesimal isometry of $g$ i.e.\ the metric $g$ is preserved by the flow of $X_0$ (cf. \cite{GW}). Therefore one can consider (at least locally) the quotient manifold 
$$
N=M/X_0
$$
and there is unique metric $\tilde g$ on $N$ such that its pullback to $\DD$ on $M$ coincides with $g$. The metric $\tilde{g}$ will be referred to as the projection of $g$. As pointed out above, the condition $h=0$ is independent of the orientation of $\DD$. Similarly, $N$ and $\tilde g$ do not depend on the sign of $X_0$.

The following theorem and its corollaries slightly generalise the results obtained in \cite{A,GW}.
\begin{theorem}\label{thm2}
Let $(\DD,g)$ be a contact sub-Riemannian or sub-Lorentzian structure on a three-dimensional manifold $M$. If $h=0$ then
the projection of $g$ to the quotient manifold $N$ determines the structure $(\DD,g)$ uniquely.
\end{theorem}
\begin{proof}
Let $N=M/X_0$ be the quotient manifold equipped with $\tilde g$. Let $(\tilde X_1,\tilde X_2)$ be an orthonormal frame on $N$. Any vector field on $N$ lifts uniquely to a vector field on $M$, tangent to $\DD$. In particular one can consider lift of the frame $(\tilde X_1,\tilde X_2)$ and get a frame $(X_1,X_2)$ of $\DD$. The frame $(X_1,X_2)$ is orthonormal for $g$, as follows from the definition of $\tilde g$. The structure functions of $(X_1, X_2)$ are special. Indeed 
$$
c_{01}^i=c_{02}^i=0,\qquad i=1,2,
$$
because $X_1$ and $X_2$ are lifts of vector fields on $N$ and $X_0$ is tangent to the fibres of $M\to N$. Moreover $c_{12}^1$ and $c_{12}^2$ are constant along $X_0$ and are actually defined by 
$$
[\tilde X_1,\tilde X_2]=c_{12}^1\tilde X_1+c_{12}^2\tilde X_2.
$$
Now, assume that there are given two structures $(\DD_1,g_1)$ and $(\DD_2,g_2)$ such that $h_1=h_2=0$ and the corresponding metrics $\tilde g_1$ and $\tilde g_2$ are isomorphic. Then one can choose orthonormal frames for $\tilde g_1$ and $\tilde g_2$ such that their structural functions coincide. It follows that the lifted frames also share the structural functions. The theorem follows from the result of E. Cartan on the equivalence of frames.
\end{proof}

\begin{corollary}\label{cor1}
If $h=0$ then the Gauss curvature of the metric $\tilde g$ on the quotient manifold $N$ equals to $\kappa$.
\end{corollary}
\begin{proof}
We use an orthonormal frame $(X_1, X_2)$ of $\DD$ defined by the lift of $(\tilde X_1, \tilde X_2)$ as in the proof of Theorem \ref{thm2} above. Then, there is no term involving $c_{0i}^j$ in $\kappa$ and the formula reduces to the Gauss curvature of $\tilde g$ computed in terms of the structural functions of $(\tilde X_1,\tilde X_2)$.
\end{proof}

\begin{corollary}\label{cor2}
The structure $(\DD,g)$ is locally equivalent to the flat structure on the Heisenberg group if and only if $h=0$ and $\kappa=0$.
\end{corollary}

\section{Einstein-Weyl geometry}

\paragraph{Definitions.}
We shall briefly recall a definition of the Einstein-Weyl structures in dimension 3 and refer to \cite{D,H,JT,K,PT} for more information on the subject. Let $M$ be a three-dimensional manifold equipped with a conformal metric $[G]$. We say that a linear connection $\nabla$ on $M$ is a Weyl connection for $[G]$ if 
$$
\nabla G=\eta\otimes G
$$
for some one-form $\eta$. The one-form is not defined uniquely by $[G]$ but depends on the representative $G$. Indeed, if one takes a different representative, given in the form $\phi^2 G$, then $\eta$ is modified by a closed form $2d(\ln\phi)$. A Weyl structure on $M$ is a pair $([G],\nabla)$. Note that, as a particular example, one can take as $\nabla$ the Levi-Civita connection for $G$. In this case $\eta=0$. In general, the connection $\nabla $ is uniquely determined by a representative $G$ and the corresponding $\eta$. Therefore, we will sometimes write that the Weyl structure is given by a pair $(G,\eta)$.

A Weyl structure is called Einstein-Weyl if it satisfies the following conformal Einstein equation 
$$
Ric(\nabla)_{sym}=\frac{1}{3}R_GG,
$$
where $Ric(\nabla)_{sym}$ is the symmetric part of the Ricci tensor and $R_G$ is the scalar curvature of $\nabla$ with respect to $G$.

\paragraph{Ricci curvature of sub-pseudo-Riemannian structures.}
Let us consider now a three-dimensional sub-Riemannian or sub-Lorentzian structure $(\DD,g)$ and assume that the Reeb vector field $X_0$ is an isometry. We have the following
\begin{proposition}\label{propEW1}
If $h=0$ then the Ricci tensor of $\nc$ in the frame $(X_1,X_2,X_0)$ is represented by the matrix 
$$
Ric(\nc)=\left( 
\begin{array}{ccc}
\kappa-\frac{c}{2} & 0 & 0\\ 
0 & \kappa-\frac{c}{2} & 0\\ 
0 & 0 & \frac{c^2}{2}
\end{array}
\right)
$$
in the sub-Riemannian case, and by the matrix 
$$
Ric(\nc)=\left( 
\begin{array}{ccc}
\kappa-\frac{c}{2} & 0 & 0\\ 
0 & -\kappa+\frac{c}{2} & 0\\ 
0 & 0 & -\frac{c^2}{2}
\end{array}
\right)
$$
in the sub-Lorentzian case.
\end{proposition}
\begin{proof}
The proof is based on computations. We shall show the details in the sub-Lorentzian case only. Let us recall that the condition $h=0$ in terms of the structural functions is equivalent to 
$$
\gamma=0,\qquad c_{01}^1=0,\qquad c_{02}^2=0.
$$
Using this and applying the Jacobi identity to vector fields $X_1$, $X_2$ and $X_0$ we get $X_0(c_{12}^1)-X_1(c_{02}^1)+c_{01}^2c_{12}^2=0$ and $X_0(c_{12}^2)+X_2(c_{01}^2)+c_{02}^1c_{12}^1=0$. By direct calculations we
compute 
$$
G^c(R(X_i,X_j)X_k,X_i)=0
$$
provided that $j\neq k$. It follows that $Ric(\nc)$ is diagonal. In order to get the diagonal terms we show that 
$$
G^c(R(X_0,X_1)X_1,X_0)=\frac{c^2}{4},\qquad G^c(R(X_0,X_2)X_2,X_0)=-\frac{c^2}{4},
$$
and the rest follows from Theorem \ref{thm1} and symmetries of the Riemann tensor.
\end{proof}

As a corollary we get the following examples of Einstein-Weyl structures.
\begin{theorem}\label{thmEW1}
Let $(\DD,g)$ be a contact sub-Riemannian or sub-Lorentzian structure on a three-dimensional manifold $M$. Assume that $h=0$ and $\kappa$ is constant and non-zero. Then the pair $([G^\kappa],\nabla^\kappa)$ defines an Einstein-Weyl structure on $M$. The Einstein-Weyl structure has the Riemannian signature if $g$ is sub-Riemannian and $\kappa>0$ and otherwise it has the Lorentzian signature.
\end{theorem}
\begin{proof}
Let us observe that $\nabla^\kappa$ is a Weyl connection for $[G^\kappa]$ as the Levi-Civita connection. The fact that the Einstein equation is satisfied follows directly from Proposition \ref{propEW1}, where we substitute $c=\kappa$.
\end{proof}

\begin{remark}
Note that all structures $([G^\kappa],\nabla^\kappa)$, with constant $\kappa$, are locally equivalent up to the sign of $\kappa$ and signatures of $g$ and $G^c$. Indeed a simple rescaling of $G^\kappa$ by a constant function gives an equivalence of $([G^\kappa],\nabla^\kappa)$ and $([G^1],\nabla^1)$ or $([G^{-1}],\nabla^{-1})$, depending on the sign of $\kappa$.
\end{remark}

\paragraph{Deformations.}
We shall show now that the Einstein-Weyl structures defined above in Theorem \ref{thmEW1} can be deformed to 1-parameter families of Einstein-Weyl structures. Let $\alpha$ be the one-form on $M$ annihilating $\DD$ and such that 
$$
\alpha(X_0)=1,
$$
where $X_0$ is the Reeb vector field, as before. We will consider Weyl structures defined by pairs $(G^c,2\epsilon c\alpha)$, where $G^c$ is an extension of $g$, $c\in\R\setminus\{0\}$ and $\epsilon\in\R$. All these structures are canonically defined by the original sub-pseudo-Riemannian structure. The Weyl connection defined by $(G^c,2\epsilon c\alpha)$ will be denoted $\nabla^{\epsilon c}_\alpha$. We have
\begin{proposition}\label{propEW2}
If $h=0$ then the symmetric part of the Ricci tensor of $\nabla^{\epsilon c}_\alpha$ in the frame $(X_1,X_2,X_0)$ is represented by the matrix 
$$
Ric(\nabla^{\epsilon c}_\alpha)_{sym}=\left( 
\begin{array}{ccc}
\kappa-\frac{c}{2}-\epsilon^2c & 0 & 0\\ 
0 & \kappa-\frac{c}{2}-\epsilon^2c & 0\\ 
0 & 0 & \frac{c^2}{2}
\end{array}
\right)
$$
in the sub-Riemannian case, and by the matrix 
$$
Ric(\nabla^{\epsilon c}_\alpha)_{sym}=\left( 
\begin{array}{ccc}
\kappa-\frac{c}{2}+\epsilon^2c & 0 & 0\\ 
0 & -\kappa+\frac{c}{2}-\epsilon^2c & 0\\ 
0 & 0 & -\frac{c^2}{2}
\end{array}
\right).
$$
in the sub-Lorentzian case.
\end{proposition}
\begin{proof}
We have $\nc\alpha=-c\alpha^1\wedge\alpha^2$, where $\alpha^i(X_j)=\delta^i_j$. Therefore $\nc\alpha$ is skew-symmetric and it follows from the formulae for the Ricci tensor of $\nabla^{\epsilon c}_\alpha$ in terms of the Levi-Civita connection of $G^c$ and the one-form $\eta$ defining the Weyl structure $(G^c,\eta)$ that only diagonal terms appear in $Ric(\nabla^{\epsilon c}_\alpha)_{sym}$ (see \cite{D,JT}). We have computed these terms directly.
\end{proof}

Our main result in this section is the following
\begin{theorem}\label{thmEW2}
Let $(\DD,g)$ be an oriented contact sub-Riemannian or sub-Lorentzian structure on a three-dimensional manifold $M$. Assume that $h=0$ and $\kappa$ is constant. If $\kappa=0$ then the pair $(G^c,2\epsilon c\alpha)$ defines an Einstein-Weyl structure on $M$ if and only if $g$ is sub-Lorentzian, $\epsilon^2=1$ and $c\in\R\setminus\{0\}$ is arbitrary. If $\kappa\neq 0$ then the pair $(G^c,2\epsilon c\alpha)$ defines an Einstein-Weyl structure on $M$ if and only if 
$$
c=\frac{\kappa}{1+\epsilon^2},
$$
in the Riemannian signature, or 
$$
c=\frac{\kappa}{1-\epsilon^2},
$$
in the Lorentzian signature, provided that $\epsilon^2\neq 1$.
\end{theorem}
\begin{proof}
We use Proposition \ref{propEW2} and deduce that the symmetric part of the Ricci tensor of $\nabla^{\epsilon c}_\alpha$ is proportional to the metric if and only if the following algebraic equation 
$$
\kappa-\frac{1}{2}c\pm\epsilon^2c=\frac{c}{2}
$$
is satisfied, where the sign next to $\epsilon^2$ depends on the signature. Thus, the Einstein equation is reduced to $c(1+\epsilon^2)=\kappa$ in the Riemannian case, or $c(1-\epsilon^2)=\kappa$ in the Lorentzian case. Therefore, if $\kappa=0$ then $g$ has to be Lorentzian and $\epsilon^2=1$ because $c\neq 0$. If $\kappa\neq 0 $ then $c=\frac{\kappa}{1\pm\epsilon^2}$ depending on the signature.
\end{proof}

If $\kappa=0$ then the family $(G^c,2c\alpha)$ defined by Theorem \ref{thmEW2} in the sub-Lorentzian case is the family of Nil Einstein-Weyl structures defined in \cite{DK}. It follows that the structures are of hyper-CR type \cite{D, DK}. The family is a deformation of the flat Lorentzian Einstein-Weyl structure. Indeed, if $c$ tends to $0$ then the structure tends to the flat one.

If $\kappa\neq 0$ then by rescaling we can assume $\kappa=1$ or $\kappa=-1$, depending on the sign of $\kappa$. The corresponding families $(G_\epsilon,\eta_\epsilon)=(G^{\frac{1}{1\pm\epsilon^2}},\frac{ 2\epsilon}{1\pm\epsilon^2}\alpha)$ for positive $\kappa$ and $(G_\epsilon,\eta_\epsilon)=(G^{\frac{-1}{1\pm\epsilon^2}},\frac{ -2\epsilon}{1\pm\epsilon^2}\alpha)$ for negative $\kappa$, defined in Theorem \ref{thmEW2}, are deformations of the structures defined in Theorem \ref{thmEW1}. Indeed, taking $\epsilon=0$ we get that $\eta_\epsilon=0$ and the corresponding Weyl connection is the Levi-Civita connection of $G^1$ or $G^{-1}$, respectively. Summarising, there are four families, depending on the sign of $\kappa$ and signature of $g$.

The structures can be easily write down in coordinates. For instance, in the Lorentzian signature and for $\kappa=1$ the family is given by $(G_\epsilon,\eta_\epsilon)$, where 
$$
G_\epsilon=-(dx-xdy)^2+dy^2+\frac{1}{1-\epsilon^2}(dz-xdy)^2,\qquad \eta_\epsilon=\frac{2\epsilon}{1-\epsilon^2}(dz-xdy).
$$
This is the Lorentzian counterpart of the Einstein-Weyl metric on the Berger sphere written in non-spherical coordinates \cite[page 96]{PT}. Our family in the Riemannian signature with $\kappa=1$ coincides with \cite[eq. (5.10)]{JT}.

\section{Dimension $2n+1$}

\paragraph{Structural functions.}
Let $M$ be a manifold of dimension $2n+1$ with a contact distribution $\DD$ equipped with a metric $g$ of arbitrary signature. We assume that there is given an orientation of $\DD$, and, as in dimension 3, we choose a local, positively oriented, orthonormal frame $(X_1,\ldots ,X_{2n})$ of $\DD$. We have that $g(X_i,X_j)=0$ for $i\neq j$ and 
$$
g(X_{i},X_{i})=s_{i},\qquad i=1,\ldots 2n
$$
where $s_{i}\in \{-1,1\}$ depends on the signature. The frame is complemented to a full frame on $M$ by the Reeb vector field, denoted $X_0$. In order to define $X_0$ we consider a one-form $\alpha$ annihilating $\DD$. It is given up to a multiplication by a non-vanishing function. However $\alpha$ can be normalised by the condition
\begin{equation}\label{norm}
(d\alpha |_{\DD})^{\wedge n}(X_1,\ldots ,X_{2n})=(-1)^{n}.
\end{equation}
The condition does not depend on the choice of a positively oriented, orthonormal frame. The Reeb vector field is uniquely defined by 
$$
X_0\in \ker d\alpha,\qquad \alpha (X_0)=1.
$$

It follows from the definition that the flow of $X_0$ preserves $\DD$, i.e.\, $[X_0,\DD]=0$. Therefore 
$$
[X_0,X_i]=\sum_{k=1}^{2n} c_{0i}^kX_k
$$
for some functions $c_{0i}^j$. Moreover, we can write 
$$
[X_i,X_j]=\sum_{k=1}^{2n} c_{ij}^kX_k + c_{ij}^0X_0
$$
and in this way we define structural functions of the frame $(X_1,\ldots,X_{2n})$. Note that $c_{ij}^0=-d\alpha(X_i,X_j)$.

\begin{remark}
There are more subtle notions of orientation of sub-pseudo-Riemannian structures based on the so-called casual decomposition of $\DD$ into its space-like and time-like subspaces \cite{G2}. However, we shall not use them, and only consider an orientation of $\DD$ itself needed in order to define the Reeb vector field.
\end{remark}

\paragraph{Canonical extension and invariatns.}
The one-form $\alpha$, used above in the definition of the Reeb field, defines an invariant skew-symmetric form on $\DD$ via the formula $d\alpha|_\DD$. We will denote it by $\omega$, i.e. 
\begin{equation}\label{omega}
\omega(X,Y)=-d\alpha(X,Y)
\end{equation}
where $X$ and $Y$ are sections of $\DD$. If $n=1$ then $\omega$ is just the volume form on $\DD$ defining the orientation.

In order to construct additional invariants one proceeds similarly to the case of dimension 3. First, one defines an invariant bi-linear symmetric form $h$ by 
\begin{equation}\label{hdimgen}
h=\frac{1}{2}\LL_{X_0}g.
\end{equation}
Then one considers extended metrics $G^c$ defined by 
$$
G^c(X_0,X_0)=c,
$$
and 
$$
G^c(X_0,\DD)=0,
$$
and gets the Levi-Civita connections $\nc$ and the corresponding curvature tensors $R^c$. The part of $R^c$ restricted to $\DD$ will be denoted $R^c_\DD$, whereas the corresponding sectional curvature of a plane $\spn\{X,Y\}\subset\DD$ will be denoted $\kappa^c_\DD(X,Y)$. The extended metrics, Levi-Civita connections $\nc$ and the associated curvatures do not depend on the orientation.

As in dimension 3, we will use $h^\sharp\colon\DD\to\DD$ defined by 
$$
g(h^\sharp(X),Y)=h(X,Y).
$$
Additionally, we will extend $g$ to a metric on exterior powers $\bigwedge^k\DD$ in the standard way.

In terms of the structural functions of a frame $(X_1,\ldots X_{2n})$ we have 
$$
\omega(X_i,X_j)=c_{ij}^0,
$$
and 
$$
h(X_i,X_j)=-\frac{1}{2}(c_{0i}^js_j+c_{0j}^is_i),
$$
where as before $s_i\in\{-1,1\}$ depends on the signature of $g$. Moreover, we have 
$$
\nc_{X_i}X_j=\frac{1}{2c}(c_{ij}^0c+c_{0j}^is_i+c_{0i}^js_j)X_0+ \sum_{k=1}^{2n}\frac{1}{2s_k}(c_{ij}^ks_k-c_{jk}^is_i-c_{ik}^js_j)X_k
$$
and 
$$
\nc_{X_i}X_0=-\sum_{k=1}^{2n}\frac{1}{2s_k}(c_{0i}^ks_k+c_{0k}^is_i+c_{ik}^0c)X_k.
$$
Additionally $\nc_{X_0}X_i=\nc_{X_i}X_0+[X_0,X_i]$ and $\nc_{X_0}X_0=0$. The following result generalises Theorem \ref{thm1}.

\begin{theorem}\label{thm3}
Let $(\DD,g)$ be a contact sub-pseudo-Riemannian structure on a $(2n+1)$-dimensional manifold. Then the sectional curvature $\kappa^c_\DD$ decomposes as follows 
\begin{equation}\label{kappacgen}
\kappa^c_\DD(X_i,X_j)=\kappa_\DD(X_i,X_j)-\frac{1}{c}g(X_i\wedge X_j,h^\sharp(X_i)\wedge h^\sharp(X_j))-\frac{3c}{4}\omega(X_i,X_j)^2,
\end{equation}
where $(X_1,\ldots, X_{2n})$ is an orthonormal frame of $(\DD,g)$ and $\kappa_\DD(X_i,X_j)$ is a quantity independent of the chosen constant $c$. In terms of the structural functions 
\begin{equation}\label{kappagen}
\begin{aligned}
\kappa_\DD(X_i,X_j)&=X_i(c_{ij}^j)s_j-X_j(c_{ij}^i)s_i-\sum_{k=1}^{2n}(c_{ij}^k)^2s_k \\
&+\sum_{k=1}^{2n}\frac{1}{4s_k}(c_{ij}^ks_k+c_{jk}^is_i-c_{ik}^js_j)^2+ \frac{1}{2}c_{ij}^0(c_{0i}^js_j-c_{0j}^is_i).
\end{aligned}
\end{equation}
\end{theorem}
\begin{proof}
The proof is reduced to computations generalising three-dimensional case. In particular \eqref{kappacgen} generalises \eqref{kappac}, and \eqref{kappagen} generalises \eqref{kappasr} and \eqref{kappasl}.
\end{proof}

\begin{remark}
Note that the sectional curvature $\kappa^c_\DD(X_i,X_j)$ determines completely the Riemann tensor $R^c_\DD$ by the well-known formula \cite[Lemma 3.3.3, page 144]{J}. By the same formula applied to $\kappa_\DD(X_i,X_j)$ we can define a $(3,1)$-tensor, denoted $R_\DD$, independent on the choice of $c$.
\end{remark}

\paragraph{Symmetric case.}
Assume that $h=0$. We consider the quotient manifold 
$$
N=M/X_0
$$
with the unique metric $\tilde g$, called projection of $g$ to $N$, such that its pullback to $\DD$ on $M$ coincides with $g$. Similarly, if $\LL_{X_0}\omega=0$ then there is the unique 2-form $\tilde\omega$ on $N$, called projection of $\omega$, such that its pullback to $\DD$ on $M$ coincides with $\omega$. As in the case of dimension 3, the condition $h=0$, and similarly $\LL_{X_0}\omega=0$, is independent of the orientation of $\DD$. We have the following generalisation of Theorem \ref{thm2}.

\begin{theorem}\label{thm4}
Let $(\DD,g)$ be a sub-pseudo-Riemannian structure on a $(2n+1)$-dimensional manifold $M$. If $h=0$ and $\LL_{X_0}\omega=0$ then the projections of $g$ and $\omega$ to the quotient manifold $N$ determine the structure $(\DD,g)$ uniquely.
\end{theorem}
\begin{proof}
The proof is a repetition of the proof in dimension 3. We shall consider an orthonormal frame $(\tilde X_1,\ldots, \tilde X_{2n})$ on $N$ and its lift $(X_1,\ldots, X_{2n})$ on $M$. Then $(X_1,\ldots, X_{2n})$ is an orthonormal frame of $\DD$ and it is easy to show that the corresponding structural functions are determined by the structural functions of the original frame on $N$ and by the projection of $\omega$.
\end{proof}

\begin{corollary}
If $h=0$ then the pullback to $\DD$ of the Riemann curvature tensor of the metric $\tilde g$ coincides with $R_\DD$.
\end{corollary}
\begin{proof}
We use an orthonormal frame $(X_1,\ldots ,X_{2n})$ of $\DD$ defined by the lift of $(\tilde{X}_1,\ldots ,\tilde{X}_{2n})$ as in the proof of Theorem \ref{thm4} above. Then, there is no term involving $c_{0k}^l$ in $\kappa _{D}(X_i,X_j)$ and the formula reduces to the sectional curvature of $\tilde{g}$ computed in terms of the structural functions of $(\tilde{X}_1,\ldots,\tilde{X}_{2n})$.
\end{proof}

\section{Contact sub-pseudo-Riemannian symmetries}\label{sec5}

\paragraph{Sub-pseudo-Riemannian isometries.}
Suppose that $(M,\DD,g)$ is a contact $(2n+1)$-dimensional sub-pseudo-Riemannian manifold. A diffeomorphism $f\colon M\to M$ is an isometry (or a sub-pseudo isometry) of $(M,\DD,g)$ if (i) $d_{q}f(\DD_q)= \DD_{f(q)}$ for every $q\in M$, (ii) $d_q f\colon\DD_{q}\to \DD_{f(q)}$ is a linear isometry of $g$. Of course, the set of all isometries of $(M,\DD,g)$ forms a group.

Suppose that, as in the previous section, $\DD$ is endowed with an orientation. Let $(X_1,\ldots,X_{2n})$ be an orthonormal positively oriented frame of $(\DD,g)$, and let $\alpha$ be the contact one-form normalised as in \eqref{norm}. By $X_0$ we denote the Reeb vector field. Finally let $(\alpha^1,\ldots,\alpha^{2n},\alpha)$ be a coframe dual to $(X_1,\ldots,X_{2n},X_0)$. Clearly 
\begin{equation}\label{system}
\begin{split}
f^*\alpha^i&=\sum_{j=1}^{2n}a_j^i\alpha^j+a_i\alpha \\
f^*\alpha&=\lambda\alpha
\end{split}
\end{equation}
for some smooth functions $a_j^i$ and $\lambda$. The normalisation condition reads 
$$
(d\alpha|_\DD)^{\wedge n}=(-1)^n\alpha^1\wedge\ldots\wedge\alpha^{2n}|_\DD
$$
which gives $\lambda=1$ and consequently $f^*\alpha=\alpha$. Using this, it is easy to show that $\alpha(f_*X_0)=1$ and $d\alpha(f_*X_0,\cdot)=0$ proving that $f_{\ast }X_0=X_0$ (note that $d\alpha$ has $1$-dimensional kernel). Consequently, if we extend $g$ to the pseudo-Riemannian metric $G=G^{1}$ by setting $G(X_{0},X_{0})=1$, then any sub-pseudo-Riemannian isometry automatically becomes an isometry of the metric $G$. This observation is independent of the choice of an orientation on $\DD$. In this way (cf. \cite{Ko}) we are led to the following

\begin{theorem}\label{thmsym1}
The set $\mathfrak{I}(M,\DD,g)$ of all isometries of $(M,\DD,g)$ is a Lie group with respect to the open-compact topology. Moreover, in the sub-Riemannian case the isotropic subgroup $\mathfrak{I}_{q}(M,\DD,g)$ of any point $q\in M$ is compact.
\end{theorem}
\begin{proof}
Indeed, $\mathfrak{I}(M,\DD,g)$ is a closed subgroup in the group of isometries of the pseudo-Riemannian manifold $(M,G)$.
\end{proof}

By the way we obtain
\begin{proposition}\label{propsym1}
Any contact sub-pseudo-Riemannian isometry $f$ is uniquely determined by two values: $f(q_0)$ and $d_{q_0}f$, where $q_0\in M$ is an arbitrarily fixed point.
\end{proposition}
\begin{proof}
The result follows from known properties of isometries in the pseudo-Riemannian geometry.
\end{proof}

Fix an isometry $f$ of $(M,\DD,g)$. Examining \eqref{system} in more detail it is easy to check that $a^{1}=\ldots=a^{2n}=0$. Moreover it is clear that $(a_j^i)_{i,j=1,\ldots,2n}\in O(l,2n-l)$, where $l$ is the index of $g$. It follows that sub-pseudo-Riemannian structures $(\DD,g)$ on $M$ are in a one-to-one correspondence with $\GG$-structures on $M$ where 
\begin{equation}\label{group1}
\GG=\left\{
\left( 
\begin{array}{cc}
A & 0\\ 
0 & 1
\end{array}
\right) \ |\ A\in O(l,2n-l)\right\}.
\end{equation}
Indeed, any such $\GG$-structure can be realised as the bundle of horizontal orthonormal frames $O_{\DD,g}(M)\to M$ with 
$$
O_{\DD,g}(M)=\left\{ (q;v_1,\ldots,v_{2n},X_{0}(q))\ |\ v_1,\ldots,v_{2n}\text{ is a }g\text{-orthonormal basis of }\DD_q\right\}.
$$
The component of identity $\mathfrak{I}_{0}(M,\DD,g)$ in the group of isometries $\mathfrak{I}(M,\DD,g)$ can be identified now with the group of all fiber preserving mappings $F\colon O_{\DD,g}(M)\to O_{\DD,g}(M)$ such that $F^*\theta =\theta$, where $\theta$ stands for the restriction to $O_{\DD,g}(M)$ of the canonical form on the bundle of linear frames on $M$.

Fix an element $(q;v_1,\ldots ,v_{2n},X_{0}(q))\in O_{\DD,g}(M)$. Thanks to Proposition \ref{propsym1} we have the embedding 
\begin{equation}\label{embedd}
\mathfrak{I}_0(M,\DD,g)\rightarrow O_{\DD,g}(M),\quad f\mapsto (f(q);d_qf(v_1),\ldots,d_qf(v_{2n}),d_qf(X_0(q))).
\end{equation}
Note that the embedding can be used to state another proof of Theorem \ref{thmsym1}.

\paragraph{Symplectic structure.}
If $f$ is an isometry of $(M,\DD,g)$ then, clearly, $f^{\ast }d\alpha =d\alpha$. In particular $f^{\ast }\omega =\omega $, where $\omega =-d\alpha |_{\DD}$, defined by \eqref{omega}, may be regarded as a symplectic form on $\DD_q$ for every $q\in M$. This leads to the following idea. If the two structures on $\DD$: $g$ and $\omega $ are compatible, meaning that there exists a symplectic basis for $d\alpha |_{\DD_{q}} $ which is orthonormal for $g$, then $(a_j^i)_{i,j=1,\ldots,2n}\in Sp(2n)\cap O(l,2n-l)$, and such sub-pseudo-Riemannian structures $(\DD,g)$ on $M$ may be viewed as reductions of $\GG$-structures with $\GG$ defined in \eqref{group1} to $\HH$-structures on $M$ with 
\begin{equation}\label{group2}
\HH=\left\{
\left( 
\begin{array}{cc}
A & 0\\ 
0 & 1
\end{array}
\right) \ |\ A\in Sp(2n)\cap O(l,2n-l)\right\}.
\end{equation}
This is the case e.g.\ for the Heisenberg group with the left-invariant sub-pseudo-Riemannian structure. More precisely, the natural left-invariant distribution $\DD$ on the Heisenberg group is, in the exponential coordinates $q=(x_1,\ldots ,x_n,y_1,\ldots ,y_n,z)$, spanned by the following fields 
\begin{equation}\label{vectfields}
X_i=\frac{\partial}{\partial x_i}-\frac{1}{2}y_i\frac{\partial}{\partial z},\quad Y_i=\frac{\partial}{\partial y_i}+\frac{1}{2}x_i\frac{\partial }{\partial z},\quad i=1,\ldots ,n.
\end{equation}
The natural left-invariant metric $g$ on $\DD$ is defined by declaring the fields $X_i$ and $Y_i$ to be orthonormal with $g(X_i,X_i)=s_i$, $g(Y_i,Y_i)=t_i$ where $s_i,t_i\in\{-1,1\}$ depending on the signature of $g$, $i=1,\ldots ,n$. Here 
$$
\alpha =dz+\frac{1}{2}\sum_{i=1}^{n}(y_idx_i-x_idy_i),
$$
so the symplectic form on $\DD_{q}$ is 
$$
\omega_{q}=-d\alpha |_{\DD_{q}}=\sum_{i=1}^{n}dx_{i}\wedge dy_{i},
$$
and $\omega (X_i,Y_j)=\delta_{ij}$, $\omega(X_i,X_j)=\omega(Y_i,Y_j)=0$, $i,j=1,\ldots ,n$.

We shall restrict now to the case of positively definite $g$ (i.e.\ $s_i=t_i=1$, $i=1,\ldots,n$) and leave other signatures to future works. If $g $ and $\omega$ are compatible then the operator $J\colon\DD\to\DD$ defined by 
$$
\omega(X,Y)=g(J X,Y)
$$
for all $X,Y\in\DD$, is a complex structure in every distribution plane $\DD_{q}$, $q\in M$, and the group $\HH$ from \eqref{group2} is isomorphic to the unitary group $U(n)$. We get the following
\begin{theorem}\label{thmsym2}
Let $(M,\DD,g)$ be a contact sub-Riemannian manifold of dimension $2n+1$. Then $\dim\mathfrak{I}(M,\DD,g)\leq (n+1)^2$.
\end{theorem}
\begin{proof}
Indeed, if $(\DD,g)$ is the left-invariant sub-Riemannian structure on the Heisenberg group as above, then it is known that $\mathfrak{I}(\R^{2n+1},\DD,g)=\R^{2n+1}\ltimes U(n)$ (see e.g.\ \cite{T}). In particular $\dim\mathfrak{I}(\R^{2n+1},\DD,g)=2n+1+n^{2}=(n+1)^{2}$ and it follows that $\dim \mathfrak{I}(M,\DD,g)\leq (n+1)^{2}$ for all contact sub-Riemannian structures such that $g$ and $\omega$ are compatible. In the general case, when $g$ and $\omega$ are not compatible, the bundle $O_{\DD,g}(M)$ reduces to the bundle of orthonormal frames such that $\omega=\sum_{i=1}^nb_i\alpha^i\wedge\alpha^{i+n}|_\DD$, where $b_i$ are certain smooth functions on $M$, called fundamental frequencies in \cite{A} ($\pm i b_i(q)$ are eigenvalues of $J_q$). Let us assume first that $b_i=\mathrm{const}$, $i=1,\ldots,n$ and let $m$ be a number of different values of $b_i$'s. Then we can decompose $\DD=\DD_1\oplus\ldots\oplus\DD_m$ where $\DD_j$, $j=1,\ldots,m$, are sub-distributions of $\DD$ corresponding to different values of frequencies. It follows that the reduced frame bundle is a principal bundle with the group $U(n_1)\oplus\cdots\oplus U(n_m)$ where $\sum_{j=1}^m n_j=n$ and $n_i=\frac{1}{2}\dim\DD_j$. Now, it is easy to see (e.g computing the Tanaka prolongation) that $\dim \mathfrak{I}(M,\DD,g)=2n+1+\sum_{j=1}^mn_j^2<(n+1)^2$. The case of non-constant $b_i $ follows from \cite{Kr}.
\end{proof}

At the end we state a theorem which in the Riemannian signature is a corollary of the classical result of Ebin \cite{E} and in the general signature is a corollary of a recent result of Mounoud \cite{Mo}.
\begin{theorem}\label{thmsym3}
Let $(M,\DD)$ be a compact contact manifold. Then for a generic pseudo-Riemannian metric $g$ on $\DD$ the group of isometries $\mathfrak{I}(M,\DD,g)$ is trivial.
\end{theorem}

\section{Appendix: Isometries in dimension $5$}

In this appendix we will compute explicitly the group of isometries for structures $(\R^5,\DD,g)$ defined by vector fields \eqref{vectfields} in dimension $5$, where $(X_1,Y_1,X_2,Y_2)$ is a $g$-orthonormal basis of $\DD$. We consider three cases: (1) $g$ is Riemannian, (2) $g(X_1,X_1)=-1$, $g(Y_1,Y_1)=g(X_2,X_2)=g(Y_2,Y_2)=1$, (3) $g(X_1,X_1)=g(X_2,X_2)=-1$, $g(Y_1,Y_1)=g(Y_2,Y_2)=1$. In all cases the metric structure is compatible with the symplectic structure as it is explained in Section \ref{sec5}. Therefore the embedding \eqref{embedd} allows to compute the corresponding isometry groups in the explicit form. The three structures are left invariant, so the corresponding group of isometries contains $5$-dimensional subgroup of left translations. In the exponential coordinates it can be represented as follows. Let $Z=\frac{\partial}{\partial z}=[X_i,Y_i]$ denote the Reeb field. The Baker-Campbell-Hausdorff formula
\begin{align*}
&\exp(x_1X_1+y_1Y_1+x_2X_2+y_2Y_2+zZ)\cdot\exp(x_1'X_1+y_1'Y_1+x_2'X_2+y_2'Y_2+z'Z)=\\
&\qquad \exp((x_1+x_1')X_1+(y_1+y_1')Y_1+(x_2+x_2')X_2+(y_2+y_2')Y_2\\
&\qquad +\left(z+z'+\tfrac{1}{2}(x_1y_1'-y_1x_1'+x_2y_2'-y_2x_2')\right)Z)
\end{align*}
gives that the isometries coming from the left translations can be written as
$$
(x_1,y_1,x_2,y_2,z)\longmapsto(x_1+t_1,y_1+t_2,x_2+t_3,y_2+t_4,z+t_5+\tfrac{1}{2}(x_1t_2-y_1t_1+x_2t_4-y_2t_3)),
$$
where $(t_1,t_2,t_3,t_4,t_5)\in\R^5$.

%
%

In case (1) the structure group $\HH$ in \eqref{group2} is the unitary group $Sp(4)\cap O(4)\simeq U(2)$ whose dimension is equal to $4$. Using suitable representation of $U(2)$ as a subgroup of $GL(4,\R)$, every $\sigma \in U(2)$ induces an isometry $(x_1,y_1,x_2,y_2,z)\longmapsto (\sigma (x_1,y_1,x_2,y_2),z)$ of the Heisenberg group. In this way we obtain the following $4$-parameter family of isometries:
\begin{align*}
&(x_1,y_1,x_2,y_2,z)\longmapsto (x_2\cos\theta_1-y_2\sin\theta_1,x_2\sin\theta_1+y_2\cos\theta_1,x_1,y_1,z),\\
&(x_1,y_1,x_2,y_2,z)\longmapsto (x_1\cos\theta_2-y_1\sin\theta_2,x_1\sin\theta_2+y_1\cos\theta_2,x_2,y_2,z),\\
&(x_1,y_1,x_2,y_2,z)\longmapsto (x_1,y_1,x_2\cos\theta_3-y_2\sin\theta_3,x_2\sin\theta_3+y_2\cos\theta_3,z),\\
&(x_1,y_1,x_2,y_2,z)\longmapsto (x_2,y_2,x_1\cos\theta_4-y_1\sin\theta_4,x_2\sin\theta_4+y_2\cos\theta_4,z).
\end{align*}
Thus, in this case $\dim\mathfrak{I}(\R^5,\DD,g)=9$.

Next, in the case (2) the structure group $\HH=Sp(4)\cap O(1,3)$ is $2$-dimensional and, in addition to left
translations, we have the following $2$-parameter family of isometries:
\begin{align*}
&(x_1,y_1,x_2,y_2,z)\longmapsto (x_1\cosh\theta_1+y_1\sinh\theta_1,x_1\sinh\theta_1+y_1\cosh\theta_1,x_2,y_2,z)\\
&(x_1,y_1,x_2,y_2,z)\longmapsto (x_1,y_1,x_2\cos\theta_2-y_2\sin\theta_2,x_2\sin\theta_2+y_2\cos\theta _2,z).
\end{align*}
Thus, in this case $\dim\mathfrak{I}(\R^5,\DD,g)=7$.

Finally, in the case (3) the structure group $\HH=Sp(4)\cap O(2,2)$ is $4$-dimensional and, in addition to left translations, we have the following $4$-parameter family of isometries:
\begin{align*}
&(x_1,y_1,x_2,y_2,z)\longmapsto (x_1\cosh \theta_1+y_1\sinh \theta _1,x_1\sinh \theta _1+y_1\cosh \theta_1,x_2,y_2,z)\\
&(x_1,y_1,x_2,y_2,z)\longmapsto (x_1,y_1,x_2\cosh \theta_2+y_2\sinh \theta _2,x_2\sinh \theta_2+y_2\cosh \theta _2,z)\\
&(x_1,y_1,x_2,y_2,z)\longmapsto (x_2\cosh \theta_3+y_2\sinh \theta _3,x_2\sinh \theta _3+y_2\cosh \theta_3,x_1,y_1,z)\\
&(x_1,y_1,x_2,y_2,z)\longmapsto (x_2,y_2,x_1\cosh \theta_4+y_1\sinh \theta _4,x_1\sinh \theta _4+y_1\cosh \theta _4,z).
\end{align*}
Thus, as in the sub-Riemannian case $\dim\mathfrak{I}(\R^5,\DD,g)=9$.

Note that in all three cases 
$$
\mathfrak{I}(\R^{5},\DD,g)=\R^{5}\ltimes \HH,
$$
where $\HH$ is the corresponding structure group.

\paragraph{Acknowledgements.} The work of Wojciech Kry\'nski has been partially supported by the Polish National Science Centre grant DEC-2011/03/D/ST1/03902.

\end{document}